\documentclass[letterpaper, 11 pt]{article}
\usepackage{fullpage,amsthm, amsmath, amssymb, amsfonts}
\usepackage[margin=1in]{geometry}
\usepackage{setspace}
\usepackage{graphicx}
\usepackage{mathtools}
\usepackage{tkz-graph}
\usepackage{tikz}
\usepackage{tikz-cd}
\usetikzlibrary{graphs}
\usetikzlibrary{automata}
\DeclarePairedDelimiter\floor{\lfloor}{\rfloor}

\newtheorem{theorem}{Theorem}
\newtheorem{proposition}{Proposition}
\newtheorem{lemma}{Lemma}
\newtheorem{question}{Question}
\newtheorem{definition}{Definition}
\newtheorem{corollary}{Corollary}
\newtheorem{remark}{Remark}
\newcommand{\newword}[1]{\textbf{\emph{#1}}}

\newcommand{\ZZ}{\mathbb{Z}}

\newcommand{\N}{\mathcal{N}}
\newcommand{\F}{\mathcal{F}}

\usepackage{lipsum}

\begin{document}

%% define your title in the usual way
\title{Necklaces and Slimes}

%% define your authors in the usual way
%% use \addressmark{1}, \addressmark{2} etc for the institutions, and use \thanks{} for contact details
\author{Suho Oh and Jina Park}
\maketitle

%% then use \addressmark to match authors to institutions here
%\address{\addressmark{1}Department of Mathematics, Texas State University, San Marcos TX \\ \addressmark{2}Department of Mathematics, Texas State University, San Marcos TX}

%% put the date of submission here
%%%%\received{\today}

%% leave this blank until submitting a revised version
%\revised{}

%% put your English abstract here, or comment this out if you don't have one yet
%% please don't use custom commands in your abstract / resume, as these will be displayed online
%% likewise for citations -- please don't use \cite, and instead write out your citation as something like (author year)
\begin{abstract}
We show there is a bijection between the binary necklaces with $n$ black beads and $k$ white beads and certain $(n,k)$-codes when $n$ is prime. The main idea is to come up with a new map on necklaces called slime migration.
\end{abstract}

\section{Introduction}
Let $n$ and $k$ be two positive integers. The main objects of this paper are the following:
\begin{itemize}
\item The set $\N_{n,k}$ of \newword{binary necklaces} (i.e. equivalent up to cyclic rotations) of length $n+k$ using $n$ black beads and $k$ white beads.
\item The set $\F_{n,k}$ of \newword{$(n,k)$-codes}, functions $f : [n] \rightarrow \ZZ^{\geq 0}$, for which their sum is $k$. The set $\F_{n,k}$ is further divided into sets $\F_{n,k,t}$ of \newword{$(n,k,t)$-codes} for $t \in \{0,\ldots,n-1\}$, where the weighted sum is $t$ modulo $n$:
$$\F_{n,k,t} := \{f | \sum_{i \in [n]} f(i) = k, \sum_{i \in [n]} if(i) \equiv t  \text{    (mod n)}\}.
$$
\end{itemize}

When $n$ is an odd positive integer, the cardinality of $\N_{n,k}$ and $\F_{n,k,0}$ is known to both equal $\frac{1}{n+k}\sum_{m|n,m|k} \varphi(m) {(n+k)/m \choose n/m}$ \cite{Chan},\cite{ACH}. It was asked in \cite{Chan},\cite{ACH} if there is a bijective proof for this. A bijective proof in the case $n$ and $k$ are coprime was given in \cite{BGP}. In this paper we construct a bijection when $n$ or $k$ is a prime number.

The proof in the case $n$ and $k$ are coprime is pretty simple: observe the fact that the weighted sum of a code changes by $k$ upon rotation. So the rotation map gives a bijection between the $\F_{n,k,t}$'s. This induces a natural bijection between $\N_{n,k}$ and $\F_{n,k,0}$. Sadly this approach does not work when $n$ and $k$ are not coprime. Regardless, one approach would be to construct a new map that is different from rotation, but still provides a nice bijection between the $\F_{n,k,t}$'s (excluding one object from $\F_{n,k,0}$). We accomplish this by using a model where several slimes in a circle all move in the same direction.

\begin{remark}
It was shown in Problem 2.11 of \cite{CP} that there is a bijection between $\F_{n,k,0}$ and the collection of out-of-debt chip-firing states on the cyclic graph $C_n$ starting with $k$ chips on a fixed vertex (going into debt before getting to the final state is allowed).
\end{remark}

\begin{remark}
It was shown in \cite{Chan} that the number of necklaces of length $n$ with at most $q$ colors and the number of codes with $n$ entries from $\{0,\ldots,q-1\}$ such that their weighted sum equals $0$ modulo $n$ is the same when $n$ and $q$ are coprime. A bijective proof was given when $q$ is a prime power. The sets $\N$ and $\F$ are different from the sets we use in our paper. We study necklaces that have $n$ black beads and $k$ white beads, whereas Chan studied necklaces that have $n$ total beads with at most $q$ colors. 
\end{remark}

%It was shown in \cite{Chan} that there is a bijection between necklaces of length $n$ and $k$ colors, and $\F_{n,k} := \union_{t} \F_{n,k,t}$. Also the set $\F_{n,k,0}$ turns out to be the collection of out-of-debt chip-firing states on the cyclic graph $C_n$ starting with $k$ chips on a vertex (going into debt before getting to the final state is allowed) \cite{CP}. In this paper we are going to show a bijection between $\N_{n,k}$ and $\F_{n,k,0}$ when $n$ is prime. 

\section{Codes and Slimes}

We always envision $[n] := \{1,\ldots,n\}$ to be having the cyclic structure of $\ZZ_n$. A cyclic interval $[i,j]$ in $[n]$ denotes $\{i,i+1,\ldots,j\}$ in $\ZZ_n$. All intervals we consider in this paper will be cyclic intervals. Let $f$ be an $(n,k)$-code. Let $m_f$ denote the largest among sum of two (position-wise) consecutive entries in $f$. For a cyclic interval $[i,j]$ in $[n]$ of size at least $2$, the corresponding sequence $f_i,\ldots,f_j$ is a \newword{weak-slime} of $f$ if $f_i+f_{i+1} = \cdots = f_{j-1}+f_j = m_f$. A weak-slime is a \newword{slime} if $f_{i-1}+f_i$ and $f_j+f_{j+1}$ are both strictly less than $m_f$ (that is, if it is inclusion-wise maximal among weak-slimes). Notice that a slime has size at least $2$ according to its definition. Any weak-slime has to have its entries alternating: it has to be of the form $a,b,a,b,\ldots,a$ or $a,b,a,b,\ldots,b$.

Given a slime $s$ of size $l$, its \newword{weight} $w(s)$ is defined as $\floor*{\frac{l}{2}}$. We denote the weight of the code $w(f)$ to be the sum of the weights of all slimes inside the code. We say that the slime is \newword{invalid} if it is the entire $[n]$ without a cutoff: the sequence $f_1,\ldots,f_n$ is an invalid slime when $f_i+f_{i+1} = m_f$ for all $i \in [n]$. We say that the code is \newword{valid} if it doesn't contain an invalid slime. A constant code (a code where $f_1=f_2=\cdots=f_n$) would have an invalid slime $[n]$ (without a cutoff) and hence be an invalid code. Since odd slimes have to be of form $a,b,\ldots,a$ with same entries on its enpoints, it is not hard to see the following:

\begin{lemma}
\label{lem:oddinv}
When $n$ is odd, the only invalid codes of length $n$ are the constant codes.
\end{lemma}

\begin{remark}
Given an $(n,k)$-code $f$, its weight $w(f)$ and weighted sum $\sum i f(i)$ (mod $n$) are different. The weighted sum changes by $k$ upon rotation. The weight on the other hand, does not change under rotation.
\end{remark}

%\begin{lemma}
%\label{lem:invalideven}
%If a nonconstant code is invalid, then $n$ is even.
%\end{lemma}

Take a look at Figure~\ref{fig:codes}. All three codes are $(11,11)$-codes. The value $m_f$, the largest among sum of two consecutive entries, is $3$ for all of them. Let us label the positions from $1$ to $11$ starting from the topmost position and going around clockwise. Then the slimes are $\{11,1\},\{2,3,4\},\{7,8,9\}$ for the first code. The weight of that code is $\floor*{\frac{2}{2}}+\floor*{\frac{3}{2}}+\floor*{\frac{3}{2}}=3$.

\begin{figure}[h!]
\centering
 \begin{tikzpicture}
\draw (0:2) arc (0:360:20mm);
%\foreach \i in {1,...,10}{
   \node[state,fill=white] at (-360/11 * 0 +126:2cm) {$2$};
	\node[state,fill=white] at (-360/11 * 1 +126:2cm) {$1$};
	\node[state,fill=white] at (-360/11 * 2 +126:2cm) {$1$};
	\node[state,fill=white] at (-360/11 * 3 +126:2cm) {$2$};
	\node[state,fill=white] at (-360/11 * 4 +126:2cm) {$1$};
	\node[state,fill=white] at (-360/11 * 5 +126:2cm) {$0$};
	\node[state,fill=white] at (-360/11 * 6 +126:2cm) {$1$};
	\node[state,fill=white] at (-360/11 * 7 +126:2cm) {$0$};
	\node[state,fill=white] at (-360/11 * 8 +126:2cm) {$3$};
	\node[state,fill=white] at (-360/11 * 9 +126:2cm) {$0$};
	\node[state,fill=white] at (-360/11 * 10 +126:2cm) {$0$};
 %     }
 \end{tikzpicture}
 \begin{tikzpicture}
\draw (0:2) arc (0:360:20mm);
%\foreach \i in {1,...,10}{
   \node[state,fill=white] at (-360/11 * 0 +126:2cm) {$1$};
	\node[state,fill=white] at (-360/11 * 1 +126:2cm) {$2$};
	\node[state,fill=white] at (-360/11 * 2 +126:2cm) {$1$};
	\node[state,fill=white] at (-360/11 * 3 +126:2cm) {$1$};
	\node[state,fill=white] at (-360/11 * 4 +126:2cm) {$2$};
	\node[state,fill=white] at (-360/11 * 5 +126:2cm) {$0$};
	\node[state,fill=white] at (-360/11 * 6 +126:2cm) {$1$};
	\node[state,fill=white] at (-360/11 * 7 +126:2cm) {$0$};
	\node[state,fill=white] at (-360/11 * 8 +126:2cm) {$2$};
	\node[state,fill=white] at (-360/11 * 9 +126:2cm) {$1$};
	\node[state,fill=white] at (-360/11 * 10 +126:2cm) {$0$};
 %     }
 \end{tikzpicture}
 \begin{tikzpicture}
\draw (0:2) arc (0:360:20mm);
%\foreach \i in {1,...,10}{
   \node[state,fill=white] at (-360/11 * 0 +126:2cm) {$1$};
	\node[state,fill=white] at (-360/11 * 1 +126:2cm) {$1$};
	\node[state,fill=white] at (-360/11 * 2 +126:2cm) {$2$};
	\node[state,fill=white] at (-360/11 * 3 +126:2cm) {$0$};
	\node[state,fill=white] at (-360/11 * 4 +126:2cm) {$3$};
	\node[state,fill=white] at (-360/11 * 5 +126:2cm) {$0$};
	\node[state,fill=white] at (-360/11 * 6 +126:2cm) {$1$};
	\node[state,fill=white] at (-360/11 * 7 +126:2cm) {$0$};
	\node[state,fill=white] at (-360/11 * 8 +126:2cm) {$1$};
	\node[state,fill=white] at (-360/11 * 9 +126:2cm) {$2$};
	\node[state,fill=white] at (-360/11 * 10 +126:2cm) {$0$};
 %     }
 \end{tikzpicture}
\caption{Some $(11,11)$-codes. The first code has $\{11,1\},\{2,3,4\},\{7,8,9\}$ as its slimes (the topmost position is indexed with $1$). If we do the forward migration on the leftmost code, we get the code in the middle. If we do the forward migration again on the middle code, we get the rightmost code. } \label{fig:codes}
\end{figure}
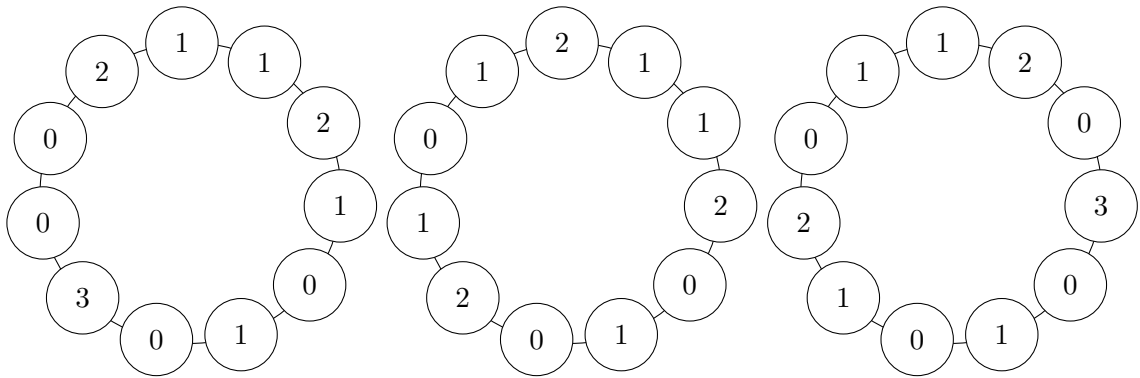

Given a valid slime that has even size, it has to be of the form
$$ a, b, \ldots ,a,b.$$
The \newword{(forward) move} on this slime transforms it to
$$ a-1, b+1, \ldots, a-1, b+1,$$
whereas the \newword{backward move} transforms the sequence to
$$ a+1, b-1, \ldots, a+1,b-1.$$
These moves are well-defined since neither $a$ nor $b$ can be $0$. Otherwise, $a$ or $b$ will be $m_f$ and the sequence can't have both $a$ and $b$ at its endpoint to be a slime. The move transforms a slime into a weak slime which can be extended to a slime by potentially adding one more position to the right (left for a backward move).

For a valid slime that has odd size, it has to be of the form
$$ a,b,\ldots,a,b,a.$$
The \newword{(forward) move} on this slime transforms the sequence to
$$ a,b-1,a+1,b-1,\ldots,a+1,b-1,a+1,$$
whereas the \newword{backward move} transforms the sequence to
$$ a+1, b-1, \ldots, a+1,b-1,a.$$
The move cuts off the leftmost element (rightmost element for a backward move) and the resulting weak slime can be extended to a slime by potentially adding one more position to the right (left for a backward move).

Given an $(n,k)$-code $f$, let $\phi_{\rightarrow}(f)$ be the code you get from $f$ by doing a forward move on all slimes of $f$ at the same time. We call this the \newword{(forward) migration} of all slimes. Similarly, let $\phi_{\leftarrow}(f)$ be the code you get from $f$ by doing a backward move on all slimes of $f$ at the same time and call this the \newword{backward migration} of all slimes. The migration changes $\sum i f(i)$ modulo $n$ by the weight of the code.

Again take a look at Figure~\ref{fig:codes}. If we do the forward migration on the first code, we get the second code. If we do the forward migration of the second code, we get the third code. If we do the backward migration on the second code, we get the first code back. 

\begin{lemma}
\label{lem:migreversible}
For any valid $(n,k)$-code $f$, we have $\phi_{\leftarrow}(\phi_{\rightarrow}(f)) = f$.
\end{lemma}
\begin{proof}
Any even sized slime $s$ of the form $a,b,\ldots,a,b$ after a forward move becomes either $a-1,b+1,\ldots,a-1,b+1$ or $a-1,b+1,\ldots,a-1,b+1,a-1$ the latter absorbing a new element to the right. In the first case it is obvious the backward move returns it back to $s$. In the second case since $a-1 + b < a+b$, the rightmost elements gets cut off and we get $s$ back as well. Similar analysis holds true for odd sized slimes.
\end{proof}

In Figure~\ref{fig:codes}, notice that the weights of all three codes are the same. It is true in general that the migration operation preserves the number of slimes and the total weight as well:

\begin{lemma}
\label{lem:migpreserve}
For any valid $(n,k)$-code $f$, a migration does not change the weight of $f$. That is, we have $w(\phi_{\leftarrow}(f)) = w(f) = w(\phi_{\rightarrow}(f))$.
\end{lemma}
\begin{proof}

Given an odd-sized slime, its size is either maintained or decreased by $1$ after a movement. Given an even-sized slime, its size is either maintained or increased by $1$ after a movement. Hence the weight of each slime is maintained after migration. 

Consider the case when we have two adjacent slimes in $f$: $f_i,\ldots,f_k$ is a slime and $f_{k+1},f_{k+2},\ldots,f_j$ is another slime. After forward migration, if the latter slime was even length then $\ldots,f_k$ and $f_{k+1},\ldots$ are still separate slimes since $f_k+f_{k+1}$ stays the same and is strictly smaller than $m_f$. If the latter slime had odd length then $f_{k+1}$ gets cut off from the slime to the right anyways. So there is no fear of two slimes merging after a migration.

We have seen that the weight of each slime is preserved and slimes do not merge nor split, allowing us to conclude that the weight of the code is preserved.

\end{proof}

%Let $c$ be the map from $\F_{n,k}$ to itself by a cyclic rotation of the entries:
%$$c(f_0,\ldots,f_{n-1}) = (f_1,\ldots,f_{n-1},f_0).$$

%We say that a map $\chi$ from $\F_{n,k}$ to itself is \newword{rotation invariant} if $c \chi = \chi c$.

%\begin{lemma}
%\label{lem:migrotation}
%For any valid $(n,k)$-code $f$, migration is rotation invariant. That is, $\phi_{\rightarrow}(c(f)) = c(\phi_{\rightarrow}(f))$.
%\end{lemma}
%\begin{proof}
%The migration map does not depend on any choice of a position on the circle and hence is rotation invariant.
%\end{proof}

Given a valid $(n,k)$-code $f$ that has weight $w(f)$ coprime with $n$, let $i(f)$ denote the inverse of $w(f)$ modulo $n$. Define $\phi(f)$ to be the map that sends $f$ to $(\phi_{\rightarrow})^{i(f)}(f)$. Combining what we have so far we get:

\begin{proposition}
\label{prop:mig}
Suppose that $w(f)$ is coprime with $n$. Then the map $\phi$ is invertible and weight preserving. Furthermore, the image of a valid code in $\F_{n,k,t}$ under $\phi$ is a valid code of $\F_{n,k,t+1}$.
\end{proposition}

\section{Bijection with necklaces}

In this section we show a bijection between $\N_{n,k}$ and $\F_{n,k,0}$ when $n$ is prime as promised. We show that coming up with a certain map on $\F_{n,k}$ implies the construction of the bijection even in the general case.

Given a code $f$, let $c$ be the rotation map that performs a cyclic rotation of the entries:
$$c(f_0,\ldots,f_{n-1}) = (f_1,\ldots,f_{n-1},f_0).$$

 We are going to express necklaces of $\N_{n,k}$ as sequences $(g_0,\ldots,g_{n-1})$ where we label the black beads $1$ to $n$ in some clockwise order, let $g_i$ count the number of white beads between black beads labeled $i$ and $i+1$. Let the sequences be equivalent under the cyclic shift (that is $(g_0,\ldots,g_{n-1}) \equiv (g_1,g_2,\ldots,g_{n-1},g_0)$). In other words, we think of a necklace in $\N_{n,k}$ as a collection of codes ${f,c(f),\ldots,c^{n-1}(f)}$. 

We dedicate $q$ to denote $\frac{n}{\text{gcd}(n,k)}$. For a code $f \in \F_{n,k,t}$, the collection $\{f, c^q(f), c^{2q}(f), \ldots, c^{n-q}(f) \}$ is called a \newword{neck-class} of $\F_{n,k,t}$. These are exactly the codes of $\F_{n,k,t}$ that are rotation equivalent to $f$.

%We define the number $q=q_{n,k}$ to be the smallest positive number such that $k q \equiv 0$ mod $n$. 

Let $\N'_{n,k}$ denote the necklaces of $\N_{n,k}$ that have period $n$. Let $\F'_{n,k,t}$ denote the codes of $\F_{n,k,t}$ that have period $n$. 

%Let $c$ be the map from $\F'_{n,k}$ to itself by a cyclic rotation of the entries:
%$$c(f_0,\ldots,f_{n-1}) = (f_1,\ldots,f_{n-1},f_0).$$

\begin{definition}
We say that a map $\chi$ from $\F'_{n,k}:= \bigcup_t \F'_{n,k,t}$ to itself is a \newword{riwi-map} if:
\begin{itemize}
\item it is a bijective map,
\item (rotation invariant) $c \chi = \chi c$, and
\item (weighted sum increasing) if $f \in \F'_{n,k,t}$ then $\chi(f) \in \F'_{n,k,t+1}$.
\end{itemize}
\end{definition}

Using a riwi-map $\chi$ we can construct a map $\sigma_{\chi}$ between $\N'_{n,k}$ and $\F'_{n,k,0}$ in the following way: for each neck-class in $\F_{n,k,0}$, fix an arbitrary representative $f$. Let $\sigma_{\chi}$ be a map from $\F'_{n,k,0}$ to $\N'_{n,k}$ that sends $c^{iq}(f)$ to $\chi^{i}(f)$ for each $0 \leq i < \frac{n}{q}$. 

\begin{theorem}
\label{thm:main}
When $\chi$ is a riwi-map, the map $\sigma_{\chi}$ is a bijection between $\F'_{n,k,0}$ and 
$\N'_{n,k}$.
\end{theorem}
\begin{proof}

We first show that the map is one-to-one. Assume for sake of contradiction that the image of some $c^{iq}(f)$ and $c^{jq}(g)$ are the same. Due to $\chi^i(f) \in \F'_{n,k,i}$, we must have $i=j$. But since $\chi$ is a bijective map, $\chi^i(f) = \chi^i(g)$ implies $f=g$ and we get a contradiction. 

Next we show that the map is onto. Pick any necklace in $\N'_{n,k}$. Choosing a position here gives a code $g$ in $\F'_{n,k,j\frac{n}{q}+i}$ for some $0 \leq j < q $ and $0 \leq i < \frac{n}{q}$. We can replace $g$ with a rotation equivalent code in $\F'_{n,k,i}$. Take the neck-class in $\F'_{n,k,i}$ containing $g$. Thanks to $\chi$ being rotation invariant, applying $(\chi^{-1})^{i}$ on the neck-class gives a neck-class in $\F'_{n,k,0}$. Pick $f$ to be the chosen representative of that neck-class. Then $c^{iq}(f)$ is mapped to $\chi^i(f)$ under $\sigma_{\chi}$ which is rotation equivalent to $g$.

\end{proof}

For codes and necklaces of period $p < n$ in $\N_{n,k}$ and $\F_{n,k,0}$ (here $p$ has to be a common divisor of $n$ and $k$), we can extend the bijection between $\N'_{\frac{n}{p},\frac{k}{p}}$ and $\F'_{\frac{n}{p},\frac{k}{p},0}$. Hence the problem of constructing a bijection between $\N_{n,k}$ and $\F_{n,k,0}$ can be reduced to the problem of finding a riwi-map on $\F'_{n,k}$'s.

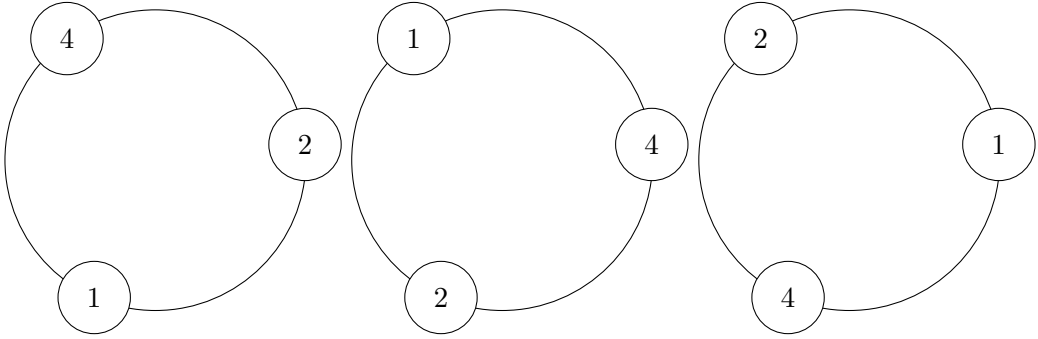
\begin{figure}[h!]

\centering
 \begin{tikzpicture}
\draw (0:2) arc (0:360:20mm);
%\foreach \i in {1,...,10}{
   \node[state,fill=white] at (-360/3 * 0 +126:2cm) {$4$};
	\node[state,fill=white] at (-360/3 * 1 +126:2cm) {$2$};
	\node[state,fill=white] at (-360/3 * 2 +126:2cm) {$1$};
 %     }
 \end{tikzpicture}
 \begin{tikzpicture}
\draw (0:2) arc (0:360:20mm);
%\foreach \i in {1,...,10}{
   \node[state,fill=white] at (-360/3 * 0 +126:2cm) {$1$};
	\node[state,fill=white] at (-360/3 * 1 +126:2cm) {$4$};
	\node[state,fill=white] at (-360/3 * 2 +126:2cm) {$2$};
 %     }
 \end{tikzpicture}
 \begin{tikzpicture}
\draw (0:2) arc (0:360:20mm);
%\foreach \i in {1,...,10}{
   \node[state,fill=white] at (-360/3 * 0 +126:2cm) {$2$};
	\node[state,fill=white] at (-360/3 * 1 +126:2cm) {$1$};
	\node[state,fill=white] at (-360/3 * 2 +126:2cm) {$4$};
 %     }
 \end{tikzpicture}
\caption{Codes corresponding to the same necklace of $\N_{3,7}$, where the top poisition is labeled with $0$. The left code is in $\F_{3,7,1}$, middle code is in $\F_{3,7,2}$ and the right code is in $\F_{3,7,0}$. Rotation increases the weighted sum by $7$ so any code of $\F_{3,7}$ can be rotated suitably many times to bring it into $\F_{3,7,0}$.}
\label{fig:37coven}
\end{figure}

It was proved in \cite{CP} that there is a bijection between $\N_{n,k}$ and $\F_{n,k,0}$ when $n$ and $k$ are coprime. Another way to think of this is that one can simply take a certain power of the rotation map as the riwi-map in this case (example in Figure~\ref{fig:37coven}):
\begin{corollary}
When $n$ and $k$ are coprime, Theorem~\ref{thm:main} gives a bijection between $\N_{n,k}$ and $\F_{n,k,0}$.
\end{corollary}
\begin{proof}
When $n$ and $k$ are coprime, every code and necklace have period $n$. So we have $\N_{n,k} = \N'_{n,k}$ and $\F_{n,k} = \F'_{n,k}$. Pick $i(k)$ to be the inverse of $k$ modulo $n$. Then $c^{i(k)}$ is a riwi-map since rotating $i(k)$ times increases the weighted sum of a code by $k \cdot i(k) = 1$.
\end{proof}

In the case $n$ is an odd prime, we use the slime migration map $\phi$ constructed in the previous section as our riwi-map (example in Figure~\ref{fig:33coven} and Figure~\ref{fig:33map}):

\begin{corollary}
When $n$ is an odd prime, Theorem~\ref{thm:main} gives a bijection between $\N_{n,k}$ and $\F_{n,k,0}$.
\end{corollary}

\begin{proof}
When $n$ is an odd prime, every codes and necklaces have period $1$ or $n$. Lemma~\ref{lem:oddinv} tells us that the only invalid codes are the constant codes. Associating the constant codes to constant necklaces, all that remains is to show that $\phi$ of Proposition~\ref{prop:mig} is a riwi-map. Bijection and weighted sum increasing are already done, so we need to check rotation invariance. The forward migration map $\phi_{\rightarrow}$ does not depend on any choice of a position on the circle since slimes are defined using sums of adjacent entries. Therefore $\phi$ is rotation invariant as well.
\end{proof}

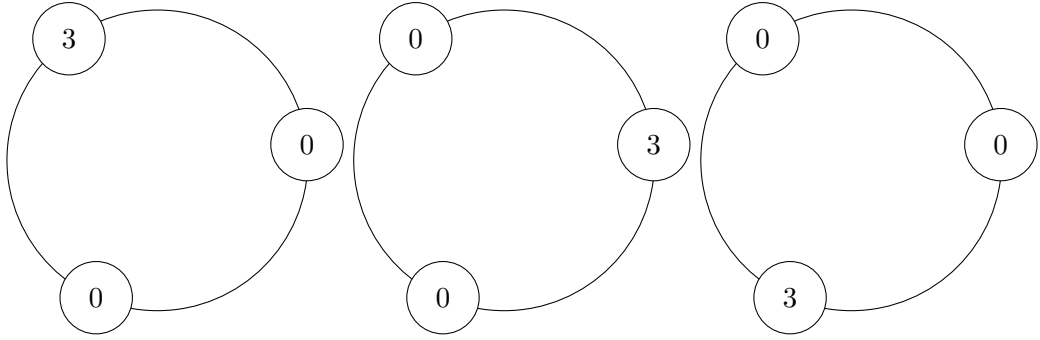
\begin{figure}[h!]
\centering
 \begin{tikzpicture}
\draw (0:2) arc (0:360:20mm);
%\foreach \i in {1,...,10}{
   \node[state,fill=white] at (-360/3 * 0 +126:2cm) {$3$};
	\node[state,fill=white] at (-360/3 * 1 +126:2cm) {$0$};
	\node[state,fill=white] at (-360/3 * 2 +126:2cm) {$0$};
 %     }
 \end{tikzpicture}
 \begin{tikzpicture}
\draw (0:2) arc (0:360:20mm);
%\foreach \i in {1,...,10}{
   \node[state,fill=white] at (-360/3 * 0 +126:2cm) {$0$};
	\node[state,fill=white] at (-360/3 * 1 +126:2cm) {$3$};
	\node[state,fill=white] at (-360/3 * 2 +126:2cm) {$0$};
 %     }
 \end{tikzpicture}
 \begin{tikzpicture}
\draw (0:2) arc (0:360:20mm);
%\foreach \i in {1,...,10}{
   \node[state,fill=white] at (-360/3 * 0 +126:2cm) {$0$};
	\node[state,fill=white] at (-360/3 * 1 +126:2cm) {$0$};
	\node[state,fill=white] at (-360/3 * 2 +126:2cm) {$3$};
 %     }
 \end{tikzpicture}
\caption{A neck-class of $\F'_{3,3,0}$. We need to map each code here to a different necklace, so rotation map isn't going to be enough.}\label{fig:33coven}
\end{figure}

\begin{figure}[h!]
\centering
 \begin{tikzpicture}
\draw (0:2) arc (0:360:20mm);
%\foreach \i in {1,...,10}{
   \node[state,fill=white] at (-360/3 * 0 +126:2cm) {$3$};
	\node[state,fill=white] at (-360/3 * 1 +126:2cm) {$0$};
	\node[state,fill=white] at (-360/3 * 2 +126:2cm) {$0$};
 %     }
 \end{tikzpicture}
 \begin{tikzpicture}
\draw (0:2) arc (0:360:20mm);
%\foreach \i in {1,...,10}{
   \node[state,fill=white] at (-360/3 * 0 +126:2cm) {$2$};
	\node[state,fill=white] at (-360/3 * 1 +126:2cm) {$1$};
	\node[state,fill=white] at (-360/3 * 2 +126:2cm) {$0$};
 %     }
 \end{tikzpicture}
 \begin{tikzpicture}
\draw (0:2) arc (0:360:20mm);
%\foreach \i in {1,...,10}{
   \node[state,fill=white] at (-360/3 * 0 +126:2cm) {$1$};
	\node[state,fill=white] at (-360/3 * 1 +126:2cm) {$2$};
	\node[state,fill=white] at (-360/3 * 2 +126:2cm) {$0$};
 %     }
 \end{tikzpicture}
\caption{Slime migration changes the leftmost code $f$ to the middle code $\phi(f)$. Doing it one more time gives the rightmost code $\phi^2(f)$. The bijection map $\sigma_{\phi}$ is going to map the codes of Figure~\ref{fig:33coven} to the necklaces we have here.}\label{fig:33map}
\end{figure}
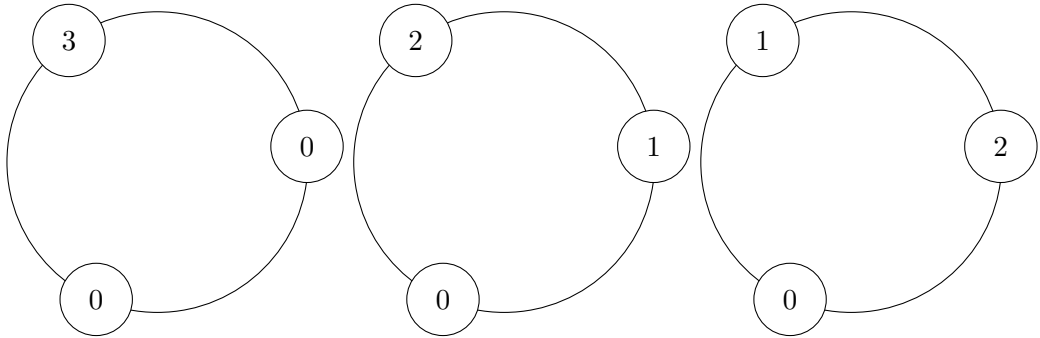

%For example, take a look at Figure~\ref{fig:33coven} and Figure~\ref{fig:33map}. Figure~\ref{fig:33coven} lists the only neck-class of $\F'_{3,3,0}$. We pick the leftmost code as the representative $f$. The middle code is $c(f)$ and the rightmost code is $c^2(f)$. Now look at Figure~\ref{fig:33map} which has $f,\phi(f),\phi^2(f)$ drawn. So $\sigma$ will map each code of Figure~\ref{fig:33coven} to the code directly below in Figure~\ref{fig:33map}.

Pretty much as a corollary we have our main result:

\begin{theorem}
We can construct a bijection between $\N_{n,k}$ and $\F_{n,k,0}$ when $n$ is prime.
\end{theorem}
\begin{proof}
We only need to consider the case when $n=2$ and $k$ is even. For a necklace of form $a,b$ with $a \geq b$, map it to the code $a,b$ if $b$ is even. Otherwise map it to $b-1, a+1$. 
\end{proof}

The reason our argument does not work directly for $n$ that is not an odd prime, is that the slime migration is not guaranteed to be a bijection between $\F'_{n,k,t}$'s. In particular, it isn't guaranteed that the period of the code stays the same after migration. For example if we do forward migration on $2,0$ we get $1,1$ and that was why we had to take care of $n=2$ case separately. Regardless, it would be interesting to see if the approach can be extended to the general case:
\begin{question}
Can one construct a bijection between $\N_{n,k}$ and $\F_{n,k,0}$, for general $n$?
\end{question}

The strategy would be to find the riwi-maps using Theorem~\ref{thm:main}. A good candidate is a modification of the slime migration map $\phi$. Notice that $\phi$ working as a riwi-map only depends on the weight of the codes being coprime with $n$, instead of what $k$ is. Since $\phi$ preserves the weight of the code, we can refine $\N'_{n,k}$ further based on the weight, and separately take care of the cases when the weight isn't coprime with $n$. Using this idea for small numbers like $n=4$ and $n=6$, it is pretty straightforward to construct a riwi-map using the fact that there are not many codes with $w(f)=2$ (being $3$ is impossible) and hence get a bijection easily.

\section*{Acknowledgement}
We would like to thank David Perkinson for introducing us to the problem. We would also like to thank the anonymous referees for their extremely helpful comments.

\bibliographystyle{plain}    % (uses file "plain.bst")
\bibliography{neck-bib}

\end{document}